\documentclass[11pt]{article}
\usepackage{amsmath,amsfonts,amssymb,amsthm,fancyhdr,bm,mathtools,enumitem,bbm}
\usepackage[hyphens]{url}
\usepackage[hidelinks]{hyperref}
\usepackage{fullpage}
\usepackage[dvipsnames]{xcolor}
\usepackage[capitalize]{cleveref}
\usepackage[color=Purple!50!white,textsize=tiny]{todonotes}
\usepackage[numbers,sort]{natbib}
\usepackage{mathtools}
\usepackage{tikz,ifthen}
\usepackage{comment}
\usepackage{dsfont}

\usetikzlibrary{decorations.markings,arrows.meta}
\tikzset{vert/.style={draw, fill=black, circle, inner sep=2pt}}

\newtheorem{theorem}{Theorem}[section]
\newtheorem{lemma}[theorem]{Lemma}

\newtheorem{proposition}[theorem]{Proposition}

\newtheorem{conjecture}[theorem]{Conjecture}
\newtheorem{question}[theorem]{Question}

\newtheorem{observation}[theorem]{Observation}
\theoremstyle{definition}
\newtheorem{definition}[theorem]{Definition}

\makeatletter
\renewenvironment{proof}[1][\proofname]{%
  \par\pushQED{\qed}\normalfont\topsep6\p@\@plus6\p@\relax
  \trivlist\item[\hskip\labelsep\bfseries #1.]\ignorespaces
}{%
  \popQED\endtrivlist\@endpefalse
}
\makeatother

\crefname{equation}{equation}{equations}
\crefname{lemma}{Lemma}{Lemmas}
\crefname{proposition}{Proposition}{Propositions}
\crefname{claim}{Claim}{Claims}
\crefname{theorem}{Theorem}{Theorems}
\crefname{conjecture}{Conjecture}{Conjectures}
\crefname{figure}{Figure}{Figures}

\AddToHook{env/lemma/begin}{\crefalias{theorem}{lemma}}
\AddToHook{env/paradigm/begin}{\crefalias{theorem}{paradigm}}
\AddToHook{env/proposition/begin}{\crefalias{theorem}{proposition}}
\AddToHook{env/corollary/begin}{\crefalias{theorem}{corollary}}
\AddToHook{env/conjecture/begin}{\crefalias{theorem}{conjecture}}
\AddToHook{env/claim/begin}{\crefalias{theorem}{claim}}
\AddToHook{env/question/begin}{\crefalias{theorem}{question}} 
\AddToHook{env/definition/begin}{\crefalias{theorem}{definition}} 
\AddToHook{env/remark/begin}{\crefalias{theorem}{remark}} 

\newlist{lemenum}{enumerate}{1}
\setlist[lemenum]{label=(\alph*), ref=\thelemma(\alph*)}
\crefalias{lemenumi}{lemma}

\newcommand{\nameditem}[1]{%
  \item \textnormal{(#1)\textbf{.}}\enspace
}

\newcommand\ab[1]{\lvert#1\rvert}

\let\leq\leqslant
\let\geq\geqslant

\title{The critical probability for percolation on finite graphs}
\author{Micha Christoph\thanks{Department of Mathematics, ETH Z\"{u}rich, 8092 Z\"urich, Switzerland. Email: \texttt{micha.christoph@math.ethz.ch}. Research supported by SNSF Ambizione Grant No. 216071.} \and Patryk Morawski\thanks{Department of Mathematics, ETH Z\"{u}rich, 8092 Z\"urich, Switzerland. Email: \texttt{patryk.morawski@math.ethz.ch}.} \and Yuval Wigderson\thanks{Institute for Theoretical Studies, ETH Z\"urich, 8006 Z\"urich, Switzerland. Email: \texttt{yuval.wigderson@eth-its. ethz.ch}. Research supported by Dr.\ Max R\"ossler, the Walter Haefner Foundation, and the ETH Z\"urich Foundation.}}
\date{}

\begin{document}

\maketitle

\begin{abstract}
    We determine the critical probability for Bernoulli bond percolation on essentially any finite graph. Namely, letting $\lambda(G)$ denote the spectral radius (maximum eigenvalue) of $G$, we prove that the critical probability is at $1/\lambda(G)$: above this probability there is typically a component of order $\Omega(\lambda(G))$, whereas below it all components are of order at most $O(\sqrt{\ab G})$. These results in particular confirm a conjecture of Krivelevich and Samotij about percolation on graphs of a given average degree, and vastly extend theorems of Bollob\'as, Borgs, Chayes, and Riordan, who proved analogous results but only for dense graphs. Our theorems are optimal in many regimes, and also demonstrate that percolation has an unexpectedly subtle behaviour on graphs whose spectral radius is roughly the square root of their maximum degree.
\end{abstract}

\section{Introduction}
Let $G$ be a graph, let $p \in [0,1]$ and consider a random subgraph $G_p$ of $G$ obtained by retaining each edge of $G$ independently with probability $p$.
This simple definition gives rise to the vast field of (bond) percolation theory, which studies the typical structure of the random graph $G_p$.
The study of percolation was initiated by Broadbent and Hammersley \cite{MR91567} in 1957, and originally focused primarily on infinite lattice graphs such as $\mathbb{Z}^d$.
The field was later dramatically transformed by the seminal paper of Benjamini and Schramm \cite{MR1423907}, who proposed studying a much more general model, in which $G$ is any infinite transitive graph.
This more general perspective has propelled the field for the past few decades, and has led to a huge number of advances and breakthroughs (e.g. \cite{MR4181032,MR4972123,MR990777,MR1169017,MR1423907}); we refer to \cite[Chapter~7]{MR3616205} for an overview of some of these developments.

A large part of percolation theory is dedicated to understanding the size and structure of the connected components of $G_p$.
For example, a fundamental observation is that if $G$ is any infinite connected graph then we can observe the following \emph{phase transition} at some \emph{critical probability} $p_c(G)$ --- if $p < p_c(G)$ then almost surely $G_p$ does not contain an infinite component, whereas if $p > p_c(G)$ then almost surely $G_p$ contains such an infinite component\footnote{This follows immediately from Kolmogorov's zero-one law, since the event that $G_p$ contains an infinite component is a tail event and its probability is monotone in $p$.}.
Many of the central questions of percolation theory revolve around understanding the value of $p_c(G)$ (as a function of $G$) and the typical structure of $G_p$ for $p$ close to $p_c(G)$.

At nearly the same time that Broadbent and Hammersley were introducing percolation, essentially the same concept was introduced to a different community under a different name.
In 1959, Gilbert~\cite{MR108839} and Erd\H{o}s and R\'enyi~\cite{MR148055,MR120167} introduced the study of random graphs.
Although the original model of Erd\H{o}s and R\'enyi was slightly different, the model that now usually bears their names is the binomial random graph $G(n,p)$, which is precisely the percolated graph $(K_n)_p$, where $K_n$ is the complete graph on $n$ vertices.
One of the most important discoveries made by Erd\H{o}s and R\'enyi is that $G(n, p)$ also undergoes a kind of phase transition: For any fixed $\varepsilon > 0$, if $p \leq (1-\varepsilon)/n$, then every component of $G(n,p)$ has order $O_{\varepsilon}(\log n)$ a.a.s.\footnote{We say that an event holds \emph{asymptotically almost surely (a.a.s.)} if its probability tends to $1$ as $n \to \infty$. When not clear from context, we will explicitly specify which parameter tends to infinity in such a.a.s.\ statements.}, whereas if $p \geq (1+\varepsilon)/n$, then $G(n, p)$ a.a.s. has a unique ``giant" component of order $\Omega_{\varepsilon}(n)$.
In other words, if one replaces the notion of finite vs.\ infinite components by ``microscopic'' vs.\ ``macroscopic'' components and the precise conditions $p < p_c$ and $p > p_c$ by the asymptotic $p \leq (1-\varepsilon)p_c$ and $p \geq (1+\varepsilon)p_c$, then one can observe a similar phenomenon of phase transition around $p_c = p_c(K_n) = 1/n$.

Given all of the above, it is very natural to study percolation on general finite graphs $G$.
Probably the first result in this direction is due to Ajtai, Koml\'os, and Szemer\'edi \cite{MR671140}, inspired by earlier work of Burtin \cite{MR505839} and Erd\H{o}s--Spencer \cite{MR534014}. 
They studied percolation in the hypercube graph $Q_d$, whose vertex set is $\{0,1\}^d$, and in which two vectors are adjacent if they differ in exactly one coordinate.
They proved that $Q_d$ also undergoes a phase transition around $p_c(Q_d) = 1/d$.
Namely, if $p < (1-\varepsilon)/d$, then a.a.s.\ every component of $(Q_d)_p$ is of order $O_{\varepsilon}(d)$, whereas if $p > (1+\varepsilon)/d$, then a.a.s.\ $(Q_d)_p$ contains a component of order $\Omega_{\varepsilon}(n)$, where $n = |Q_d| = 2^d$.
The technique introduced by Ajtai, Koml\'os and Szemer\'edi, combining a local analysis based on branching processes with a global argument using the expansion properties of $Q_d$, has further proved extremely influential.

Given the two positive examples, it is natural to ask whether one can show a similar phenomenon for all graphs $G$, or at least for $d$-regular graphs, noting that both $K_n$ and $Q_d$ are regular.
The answer is a quick no --- for example, if we take $G$ to be a disjoint union of many copies of $K_{d+1}$, then clearly we will never see the supercritical phase with a component of size $\Omega(\ab G)$ in $G_p$.
There are two solutions to this apparently serious problem.

The first, and most well-studied, one is to impose some global ``robust connectivity" properties on $G$ (as used for example by Ajtai, Koml\'os and Szemer\'edi in the study of $Q_d$), which ensure that there is a supercritical phase.
The property one naturally imposes is some variant of expansion, expressed either as a spectral condition \cite{MR2020308}, a combinatorial isoperimetry condition \cite{MR4777865,2308.10267,MR2073175}, or a structural condition which roughly implies one of the former \cite{MR4827416,MR4730907,MR4998908,MR4665636,MR4801596}.

An alternative approach, if we really want to consider \emph{all} graphs, is to weaken our expectation for a phase transition.
Namely, the above example suggests it may be more natural to call a component ``giant'' when its size is on the order of $d$, rather than $n$, which is the approach we want to take in this paper.
In this direction, Krivelevich and Sudakov \cite[Theorem~4]{MR3085765} proved the supercritical statement for regular graphs; if $G$ is a $d$-regular graph and $p \geq (1+\varepsilon)/d$, then a.a.s.\ $G_p$ contains a connected component of order $\Omega_\varepsilon(d)$. 
% This is now also a standard consequence of the DFS method introduced in \cite{MR2946076}. 

To complement this with a subcritical statement, one might want to say that if $p \leq (1-\varepsilon)/d$, then a.a.s.\ every component of $G_p$ has order $o(d)$.
Unfortunately, the same example of a disjoint union of many cliques shows that this cannot hold in general.
If $d$ and $0<p \leq (1-\varepsilon)/d$ are fixed, and the number of copies of $K_{d+1}$ is taken sufficiently large, then with overwhelming probability there will be some $K_{d+1}$ all of whose edges survive into $G_p$.
In particular, $G_p$ will have a component of size $d+1$.
Nevertheless, it is easy to show that having too many vertices is basically the only obstruction; concretely, a simple branching process argument (see, e.g., \cite[Theorem~2]{MR4827416} or \cite[Proposition~1]{MR2583058}) shows that if $G$ is $d$-regular and $p \leq (1-\varepsilon)/d$, then a.a.s.\ the largest component of $G_p$ has order $O_{\varepsilon}(\log {|G|})$, and in particular $o(d)$ as long as $|G| = 2^{o(d)}$.

Combining the supercritical result of Krivelevich and Sudakov \cite[Theorem~4]{MR3085765} with the standard subcritical estimate, we obtain the following satisfactory picture for finite $d$-regular graphs.
\begin{theorem}\label{thm:regular graphs}
    Let $\varepsilon > 0$, and let $G$ be any finite $d$-regular graph.
    \begin{enumerate}[label={\rm \bf \Roman*}, ref=(\Roman*)]
        \nameditem{Supercritical phase} If $p\geq (1+\varepsilon)/d$, then $G_p$ a.a.s.\ has a component of order $\Omega_\varepsilon(d)$.
        \nameditem{Subcritical phase} If $p\leq (1-\varepsilon)/d$, then a.a.s.\ all components of $G_p$ have order $O_{\varepsilon}(\log {|G|})$.\label{it:d-regular subcritical}
    \end{enumerate}
\end{theorem}

\subsection{Main results: percolation on general finite graphs}
Our goal in the present paper is to extend the picture from \cref{thm:regular graphs} to arbitrary, not necessarily regular, finite graphs.
Our first main result yields the same supercritical behaviour for graphs of average degree $d$.

\begin{theorem}\label{thm:main avg degree}
    For every $\varepsilon > 0$ there exists $\gamma > 0$ such that the following holds.
    If $G$ is a graph of average degree $d(G) > 0$ and $p \geq (1+\varepsilon)/d(G)$, then
    \[
        \Pr[G_p \text{ has a component of order at least }\gamma d(G)] = 1 - o(1),
    \]
    where $o(1)$ tends to $0$ as $\ab{G} \to \infty$.
\end{theorem}

\noindent Krivelevich and Samotij \cite{MR3177525} observed that taking $p\geq (c+\varepsilon)/d(G)$ for an explicit constant $c\approx1.6$ yields the same conclusion, and conjectured that the constant could be improved all the way to $1$. This question was later reiterated in \cite{2308.10267}, and the theorem above answers it affirmatively.
% \yuval{consider rewriting, KS technically ask for a cycle and not just a component}

The next natural question to ask is whether $1/d(G)$ is a critical probability for any finite graph $G$, as $1/d$ is for all $d$-regular graphs.
Our next result shows that the answer is no: the true value of $p_c(G)$ is smaller than $1/d(G)$ essentially whenever $G$ is ``far from regular''.
In order to state it, we need one more piece of terminology.

Given a graph $G$, we denote by $\lambda(G)$ its \emph{spectral radius}, that is, the largest eigenvalue of its adjacency matrix $A_G$.
This can be thought of as a certain ``smoothed'' version of the average degree of $G$, which takes into account both the degree sequence and the geometry of $G$.
For example, it is a standard exercise in spectral graph theory that $\lambda(G) \geq d(G)$, with equality if and only if $G$ is regular.

Our main theorem states that, informally speaking, $1 / \lambda(G)$ is the critical probability for any finite graph $G$, although its conclusions are slightly weaker than those in \cref{thm:regular graphs} for regular graphs.
We remark, however, that as we will shortly discuss, this reflects the reality and is not simply an artefact of our methods.
% \micha{If $\gamma'$ is big, replace it with big constant}
\begin{theorem}\label{thm:main lambda}
    For every $\varepsilon, \delta > 0$ there exist $\gamma, L > 0$ such that the following holds for any finite graph $G$.

    \begin{enumerate}[label={\rm \bf \Roman*}, ref=(\Roman*)]

        \nameditem{Supercritical phase}\label{it:lambda supercritical no assumptions}
        If $p \geq (1+\varepsilon)/\lambda(G)$ then
        \[
            \Pr[G_p \text{ has a component of order at least } \gamma \lambda(G)] \geq 1-\delta.
        \]

        \nameditem{Subcritical phase}\label{it:lambda subcritical no assumptions}
        If $p \leq (1-\varepsilon)/\lambda(G)$ then
        \[
            \Pr[G_p \text{ has a component of order at least } L\sqrt{\ab G}] \leq \delta.
        \]

    \end{enumerate}
\end{theorem}

The fact that the spectral radius of $G$ is related to the percolation on $G$ is, at this point, somewhat to be expected.
For example, it is known that the spectral radius of the mean progeny matrix determines the survival probability of multi-type branching processes (see, e.g., \cite[Chapter~V]{MR373040}).
Even more closely related to our work is a remarkable theorem of Bollob\'as, Borgs, Chayes, and Riordan~\cite{MR2599196}, which shows that the critical probability for graphs $G$ whose average degree is linear in the number of vertices is $1/\lambda(G)$.
Later, Chung, Horn, and Lu~\cite{chunghorn} proved related results, which also show a phase transition around $1/\lambda(G)$ for a restricted class of sparser graphs $G$.\footnote{Their results are stated in terms of the second order average degree, which, however, asymptotically coincides with $\lambda(G)$ under their assumptions on the graph.}

As mentioned above, \cref{thm:main lambda} gives a slightly weaker picture than the one in \cref{thm:regular graphs}.
Firstly, in the supercritical regime, we only get a component of size $\Omega(\lambda(G))$ with probability arbitrarily close to one, but not tending to one.
This in particular makes \cref{thm:main lambda} incomparable\footnote{Recall that $\lambda(G)\geq d(G)$ for any graph $G$, hence \cref{thm:main lambda} is ``morally'' stronger than \cref{thm:main avg degree}: we are percolating at a lower value of $p$, and obtaining a stronger lower bound on the size of the component, but at the expense of a weaker estimate on the probability of this event.} to \cref{thm:main avg degree}.
Secondly, in the subcritical regime, the largest component has size $O(\sqrt{\ab{G}})$, not $O(\log {\ab{G}})$ as for regular graphs.
Quite surprisingly, both of these drawbacks are necessary, as we now discuss.

The graph showing that weakening both the supercritical and the subcritical statements in \cref{thm:main lambda} is necessary is the complete bipartite graph $K_{s, t}$.
Here, we want to think of $s$ as constant chosen sufficiently large compared to $\gamma$, and of $t$ as tending to infinity.
To explain this example, let $S$ and $T$ be the corresponding sides in the graph $K_{s,t}$ and, since it is well known that $\lambda(K_{s,t}) = \sqrt{st}$, let $p = C/\sqrt{st}$ for some constant $C > 0$ (we will think of $C\leq 1-\varepsilon$ for analysing the subcritical case, and $C\geq 1+\varepsilon$ for the supercritical case).
By the Chernoff and union bounds, it is easy to see that a.a.s.\ as $t \to \infty$ in $(K_{s,t})_p$ each vertex in $S$ will have $(1+o(1))C\sqrt{t/s}$ neighbours in $T$ (recall that we think of $s$ as fixed).
On the other hand, with probability $\exp(-\Theta(C^2s))$, which is a constant, each vertex in $T$ will have at most one neighbour in $S$.
In particular, with small but constant probability, the largest component of $(K_{s,t})_p$ will have size $(1+o(1))C\sqrt{t/s}$.
If we pick $s$ to be a large enough constant compared to $C$ and $\gamma$, then this will be less than $\gamma \lambda(G)$, which shows the optimality of the supercritical statement in \cref{thm:main lambda}: we truly cannot obtain a probability tending to $1$, but only arbitrarily close to $1$.
For the subcritical statement, we notice that with probability at least $\Omega((C^2/s)^s)$ all the vertices in $S$ will lie in the same connected component of $(K_{s,t})_p$.
Indeed, we can fix an arbitrary tree $F$ on the vertex set $S$, and notice that for every $xy \in E(F)$ the probability that in the percolated graph $(K_{s, t})_p$ the vertices $x$ and $y$ have a common neighbour in $T$ is at least roughly $C^2/s$. It follows by the Harris inequality that all the vertices in $S$ lie in the same component with probability at least $\Omega((C^2/s)^s)$.
In particular, again with a small but constant probability, the largest component of $(K_{s, t})_p$ will have size at least $s \cdot (1+o(1))C\sqrt{t/s}$.
If we again set $s$ large enough compared to $C$ and $L$, then this will be more than $L\sqrt{\ab{G}}$.

It turns out that what goes wrong in this example is that the spectral radius of $K_{s, t}$ is roughly the square root of its maximum degree.
In fact, our final result shows that this is basically the only obstruction to obtaining a sharper picture for general graphs.
\begin{theorem}\label{thm:main large lambda}
    For every $\varepsilon > 0$ there exist $\gamma, L, c >0$ such that the following holds. 
    Let $K > 0$ and let $G$ be a finite graph with $\lambda(G)\geq  K\sqrt{\Delta(G)}$.
    Then,
    \begin{enumerate}[label={\rm \bf \Roman*}, ref=(\Roman*)]
            \nameditem{Supercritical phase}\label{it:large lambda supercritical}
            If $p \geq (1+\varepsilon)/\lambda(G)$ then
            \[
                \Pr[G_p \text{ has a component of order at least } \gamma \lambda(G)] \geq 1-\exp(-cK).
            \]

            \nameditem{Subcritical phase}\label{it:large lambda subcritical}
            If $p \leq (1-\varepsilon)/\lambda(G)$ then as $\ab G \to \infty $ we have
            \[
                \Pr\left[G_p \text{ has a component of order at least } \frac{L\sqrt{\ab G} \log {\ab G}}{ K}\right] = o(1).
            \]
    \end{enumerate}
\end{theorem}
\noindent In particular, if we think of $K$ as tending to infinity, we obtain a.a.s.\ statements in both the supercritical and subcritical phases.

We note that taking $G = K_{s, t}$ shows that the supercritical statement is optimal, up to the dependence on $K$.
By the same example, the subcritical statement is optimal for some pairs $(K, \lambda(G))$. It can, however, be improved if $\lambda(G)$ is small compared to $\ab{G}$; we refer the interested reader to the discussion in \cref{section:concluding remarks}.
We also note that \cref{thm:main large lambda} immediately recovers the result of Bollob\'as, Borgs, Chayes and Riordan~\cite{MR2599196} for dense graphs, for which both $\lambda(G)$ and $\Delta(G)$ are of order $\Theta(\ab G)$.

\vspace{10px}

The remainder of this paper is organized as follows.
First, the rest of this section introduces our notation and outlines the proofs of the supercritical statements.
Section~\ref{section:preliminaries} establishes various preliminary statements used later and Section~\ref{section:component_of_a_vertex} analyses the behaviour of the connected component of a fixed vertex in the percolated graph $G_p$. 
We prove Theorems \ref{thm:main avg degree}, \ref{thm:main lambda} and \ref{thm:main large lambda} in Section~\ref{section:proofs of theorems} and conclude with possible research directions in Section~\ref{section:concluding remarks}.

\subsection{Notation}\label{section:notation}
For a positive integer $n$, we denote by $[n]$ the set $\{1,2,\dots,n\}$.
For a graph $G$ we write $|G|$ for the number of vertices in $G$, $e(G)$ for the number of edges in $G$, $d(G)$ for its average degree, and $\Delta(G)$ for its maximum degree.
For a vertex $x \in V(G)$, we let $N_G(x)$ denote the neighbourhood of $x$, and $d_G(x)=\ab{N_G(x)}$ its degree.
For two disjoint sets $S, T \subseteq V(G)$, we write $E_G(S, T)$ to denote the set of edges in $G$ that go between $S$ and $T$, i.e., $E_G(S, T) = \{\{u, v\} \in E(G): u \in S, v \in T\}$. We also denote by $E_G(S)$ the set of edges with both endpoints in $S$.
We sometimes omit the subscript and write $N(x)$ for $N_G(x)$, etc., if the underlying graph is clear from the context.
For a vector $w \in \mathbb{R}^{V(G)}$ and a subset $S \subseteq V(G)$ we write $w(S)$ for $\sum_{x \in S} w(x)$, and denote by $G-S$ the induced subgraph of $G$ obtained by deleting the vertices in $S$. If $X_1,\dots,X_k$ are random variables, we denote by $\sigma(X_1,\dots,X_k)$ the $\sigma$-algebra they define.
Throughout, we omit floor and ceiling signs whenever doing so does not affect the argument.

\subsection{Proof outline}\label{section:proof outline}
To outline the general idea of our proofs, let us focus on the supercritical part of \cref{thm:main large lambda}.
That is, we fix $\varepsilon > 0$, let $G$ be a graph and $p \geq (1+\varepsilon)/\lambda(G)$.

The general strategy of our proof is as follows.
First, we fix a vertex $v$ and we would like to understand how the size of the connected component $\mathcal{C}_v$ of $G_p$ containing $v$ behaves.
In particular, we would like to show that with constant probability, say $c$, this connected component has size at least $\gamma\lambda(G)$.
Moreover, if this is not the case, then we would like to look at the graph $G_p$ with this component removed and repeat the argument there.
After repeating this $k = \omega(1)$ times, assuming that the remaining graph still behaves supercritically,  we would get that the probability that $G_p$ does not contain a large connected component is at most $(1-c)^k = o(1)$.

To understand how the size of $\mathcal{C}_v$ behaves, let us first sketch the argument of Krivelevich and Sudakov \cite{MR3085765} that works when $G$ is $d$-regular, in which case $\lambda(G) =d$.
To that end, let $Y_1, \dots, Y_{e(G)}$ be independent random variables, each taking value $1$ with probability $p$ and $0$ otherwise.
We want to reveal $\mathcal{C}_v$ one edge at a time using the following random process taking these variables as an input.
We will start with $\mathcal{C}_v= \{v\}$ and in each step $t$ we will consider a new edge $e \in E(G)$ connecting $\mathcal{C}_v$ to the rest of the graph.
If $Y_t = 1$, then this edge is present in $G_p$ and we will add the corresponding endpoint to $\mathcal{C}_v$.
Otherwise, this edge is not present, and we proceed to the next edge.
The process stops when there are no more unrevealed edges connecting $\mathcal{C}_v$ to the rest of the graph.

Now suppose that $|\mathcal{C}_v| = C < \gamma d$ and so the process stops at some step $t$.
Then, among $Y_1, \dots, Y_t$ we have seen exactly $C-1$ ``yeses", as each ``yes" increases the order of $\mathcal{C}_v$ by exactly one.
On the other hand, we must have seen a ``no" for each edge of $G$ that goes from $\mathcal{C}_v$ to the rest of $G$.
As $G$ is $d$-regular, there are at least $C \cdot (d - C) \geq (1-\gamma)Cd$ such edges, and hence $t \geq (1-\gamma)Cd$. Thus, at most a $1/((1-\gamma)d)$-fraction of the random variables $Y_1,\dots,Y_t$ are ``yes''.
However, if $\gamma$ is sufficiently small compared to $\varepsilon$, a Chernoff bound shows that this is rather unlikely to happen in our sequence $Y_1, \dots, Y_t$, since $p \geq (1+\varepsilon)/d$.
In fact, by analysing these probabilities one can show that with constant probability $\mathcal{C}_v$ has size at least $\gamma d$ --- and that if not, then a.a.s.\ $|\mathcal{C}_v| \ll \log d$, say.
In particular, in the latter case, after removing $\mathcal{C}_v$ from $G$ the average degree basically doesn't decrease, and since we have not revealed any edge of $G_p-\mathcal{C}_v$, we can then repeat the argument inside there.

Unfortunately, this argument inherently required that every single set $S \subseteq V(G)$ of size at most $\gamma\lambda(G)$ sends at least $(1-\gamma)\lambda(G)|S|$ edges to the rest of the graph, which can be far from true for general graphs.
To circumvent that, we consider a weighted version of this process --- where the weight vector is an eigenvector $w \in [0,1]^{V(G)}$ of $G$ corresponding to $\lambda(G)$.
Using the definition of an eigenvalue, we have for each set $S \subseteq V(G)$ of size at most $\gamma\lambda(G)$ that $\sum_{x \in S}\sum_{y \in N(x) \setminus S}w(y) \geq (1-\gamma)\lambda(G)w(S)$.
%\yuval{I guess that ``small'' means less than $\gamma \lambda(G)$ which is where the $1-\gamma$ comes from? If so we should spell that out} 
That is, while the set might not expand in the usual definition, the weight of each small set expands by roughly $\lambda(G)$.
We note that the idea of using the eigenvector as a weight vector guiding the exploration has already appeared in \cite{MR1062053}, to study percolation on a certain type of tree, and has also later been used in \cite{nelson2017probability}.

With this weight vector in hand, we can pick $v\in V(G)$ with $w(v) = 1$ and explore $\mathcal{C}_v$ in the same way as described above.
Now, each query will have some weight $w(y)$ of the vertex $y$ that we might potentially add to $\mathcal{C}_v$ in a given step.
When the process stops with $|\mathcal C_v|\leq \gamma \lambda(G)$, the weight of ``yeses'' is exactly $w(\mathcal{C}_v) -1$, while we can show that the weight of ``noes'' is at least $\sum_{x \in \mathcal{C}_v}\sum_{y \in N(x) \setminus \mathcal{C}_v}w(y) \geq (1-\gamma)\lambda(G)w(\mathcal{C}_v)$.
Hence, if the process stops early, the weight of ``yeses'' must have been much smaller than expected.
By a similar analysis to the one above, we can in fact show that with constant probability $\mathcal{C}_v$ grows to size at least $\gamma\lambda(G)$, and if not, then a.a.s.\ $w(\mathcal{C}_v)$ is small.

The final missing piece is therefore to show that if $w(\mathcal{C}_v)$ is small, then the spectral radius of $G$ won't decrease significantly when we remove $\mathcal{C}_v$.
Given the subtleties discussed in the introduction, it is, however, not surprising that such a statement is not true in general (consider, for example, $G=K_{1, n}$).
Nonetheless, we can still prove such a statement under the assumption that $\lambda(G) \gg \sqrt{\Delta(G)}$, and make the strategy work in this case, thus proving \cref{thm:main large lambda}\ref{it:large lambda supercritical}.
In other cases we resort to other, related, methods.
For example, to prove \cref{thm:main avg degree}, we additionally consider the original unweighted process, to argue that while $\lambda(G)$ might decrease, the average degree will in fact not decrease significantly.
This will be sufficient if $p \geq (1+\varepsilon)/d$.

\section{Preliminaries}\label{section:preliminaries}
In this section, we establish the auxiliary results used in the proofs of the main theorems.
We first develop the use of the eigenvector as a weight function guiding the exploration process.
Then, we prove a concentration inequality, which will replace the Chernoff inequality in the analysis of the exploration process, as discussed above.

\subsection{The eigenvector as a weight vector}
As described above, we would like to use the eigenvector $w$ corresponding to $\lambda(G)$ as a weight vector guiding the exploration of $G_p$.
For that, we always assume that $w$ is nonnegative and normalized, as in the following proposition, which is a simple consequence of the Perron--Frobenius theorem (see, e.g., \cite[Section~2.2]{BrouwerHaemers}).
\begin{proposition}\label{prop:perron_frobenius}
    Let $G$ be a non-empty graph with adjacency matrix $A_G$. Then%\yuval{I don't think the notation $\|A_G\|_2$ is quite right. To me this would mean the $\ell_2$ norm of $A_G$ when viewed as a vector, i.e.\ the Frobenius norm. But we mean the $2\to 2$ operator norm}
    \[
        \lambda(G)
        =
        \max_{x\in\mathbb{R}^{V(G)}\setminus\{0\}}
        \frac{x^T A_Gx}{x^Tx}.
    \]
    Moreover, there exists a vector $w\in[0,1]^{V(G)}$ such that
    \[
        A_Gw=\lambda(G)w
        \qquad\text{and}\qquad
        \lVert w\rVert_\infty=1.
    \]
\end{proposition}
\noindent In what follows, we will use such a vector $w$ without further reference to this proposition.

Another ingredient described in the proof outline is the statement that every small set expands in weight by a factor of roughly $\lambda(G)$.
This will be given by the following lemma.
\begin{lemma}[Small sets expand in weight]\label{lemma:small_sets_expand}
    Let $G$ be a graph, let $0 < \gamma \leq 1$ and $\lambda >0$. Moreover, let $w \in [0,1]^{V(G)}$ be such that for each $x \in V(G)$ we have $\sum_{y \in N(x)} w(y) \geq \lambda w(x)$.
    Then, for every $S \subseteq V(G)$ with $|S| \leq \gamma\lambda$ it holds that
    \[
        \sum_{x \in S} \sum_{y \in N(x) \setminus S} w(y) \geq \left(1-\gamma\right)\lambda w(S).
    \]
\end{lemma}
\begin{proof}
    Since $w$ is nonnegative, for every $x \in S$ we have
    \[
        \sum_{y \in N(x) \setminus S} w(y) = \left(\sum_{y \in N(x)}w(y)\right) - \left( \sum_{y \in N(x) \cap S} w(y)\right) \geq \lambda w(x) - w(S),
    \]
    where for the final step we used our assumption on $w$.
    In particular, using that $|S| \leq \gamma\lambda$, we get
    \[
        \sum_{x \in S}\sum_{y \in N(x) \setminus S} w(y) \geq \sum_{x \in S} \bigl(\lambda w(x) - w(S) \bigr) = \lambda w(S) - |S|w(S) \geq \bigl(1- \gamma \bigr)\lambda w(S),
    \]
    as required.
\end{proof}

As sketched above, in the proof of \cref{thm:main large lambda}\ref{it:large lambda supercritical}, we would like to repeatedly reveal a connected component of a vertex and remove it from the graph, until we find a component of size $\Omega(\lambda(G))$.
To make this work, we need to argue that after removing a few small components the remaining graph still behaves supercritically, i.e., that its spectral radius has not decreased significantly compared to that of $G$.
The following lemma says that indeed, under the assumption that $\Delta(G) \ll \lambda(G)^2$, removing a set of small weight cannot decrease the spectral radius by much.
We remark that without this assumption the statement is not true in general, as the example of $K_{1,n}$ illustrates.

\begin{lemma}[Removing small sets does not decrease the spectral radius]\label{lemma:spectral_radius_stability}
    Let $G$ be a graph and let $w \in [0,1]^{V(G)}$ be an eigenvector of $G$ corresponding to the eigenvalue $\lambda(G)$ such that $\lVert w\rVert_{\infty} = 1$.
    Assume moreover that $\lambda(G) \geq K\sqrt{\Delta(G)}$ for some $K >0$.
    Then, for each $S \subseteq V(G)$ we have that
    \[
        \lambda(G -S) \geq \left( 1- \frac{2w(S)}{K^2} \right) \lambda(G).
    \]
\end{lemma}
\begin{proof}
    We will in fact show that
    \[
        \lambda(G-S) \geq \left( 1- \frac{2w(S)}{w^Tw} \right)\lambda(G).
    \]
    Then we can lower bound $w^Tw$ using that the maximum degree of $G$ is not large.
    More specifically, letting $v$ be a vertex with $w(v) = 1$ and $d(v) \leq \Delta(G) \leq  \lambda(G)^2/K^2$, we get that
    \[
        w^Tw \geq \sum_{y \in N(v)} w(y)^2 \geq d(v) \cdot \left( \sum_{y \in N(v)} \frac{w(y)}{d(v)} \right)^2 = \frac{\lambda(G)^2}{d(v)} \geq K^2,
    \]
    where we used the Cauchy--Schwarz inequality.

    To prove the claim, we fix a set $S \subseteq V(G)$ and let $u$ denote the restriction of $w$ to $V(G) \setminus S$.
    We first note that if $u = 0$, then $w^Tw = \sum_{x \in S} w(x)^2 \leq \sum_{x \in S} w(x) = w(S)$.
    In particular, since $\lambda(G-S) \geq 0$, the statement holds vacuously.

    We can therefore assume that $u \neq 0$ and want to show that $u^TA_{G-S}u$ does not decrease much compared to $w^TA_Gw$.
    To that end, we note that
    \[
       \lambda(G) = \frac{w^TA_Gw}{w^Tw} \hspace{10px} \text{and} \hspace{10px} \lambda(G-S) \geq \frac{u^TA_{G-S}u}{u^Tu} \geq \frac{u^TA_{G-S}u}{w^Tw}.
    \]
    Moreover, by using that $0 \leq w \leq 1$, we get
    \[
        w^TA_Gw - u^TA_{G-S}u \leq 2\sum_{x \in S}w(x) \sum_{y \in N_G(x)} w(y) = 2\sum_{x \in S} \lambda(G) w(x)^2 \leq 2\lambda(G)w(S).
    \]
    Hence, by combining the above observations we get
    \[
        \lambda(G) - \lambda(G-S) \leq \frac{w^TA_Gw - u^TA_{G-S}u}{w^Tw} \leq \lambda(G) \frac{2w(S)}{w^Tw},
    \]
    as desired.    
\end{proof}

Finally, we turn to the subcritical regime, where we would like to show that
a.a.s.\ $G_p$ does not contain a large connected component. 
Following the approach outlined above, one might take $w\in[0,1]^{V(G)}$ to be a nonnegative eigenvector corresponding to $\lambda(G)$ and use that each set does not expand by more than a $\lambda(G)$-factor in weight.
Hence, if there is a component with large weight in $G_p$, then during its exploration we must have seen many more ``yeses'' than expected.
This way we can show that a.a.s.\ $G_p$ does not contain a connected component with large weight.
Unfortunately, however, an upper bound on the weight does not translate into a bound on the size, as $w(v)$ might be arbitrarily small for a vertex $v \in V(G)$, so this is not good enough for us.

To overcome this difficulty, we show that we can in fact construct another weight vector $w$, for which it still holds that each set does not expand by much more than $\lambda(G)$ but where we also have a lower bound on $w(v)$ for each vertex.
This will allow us to prove \cref{thm:main large lambda}\ref{it:large lambda subcritical} by making use of the same approach and later translating the upper bound on the weight of each component into an upper bound on its size.
The suitable weight vector will be given by the following lemma, and is defined as a weighted count of walks starting at a given vertex.
We remark that we will use a different argument for the subcritical phase in \cref{thm:main lambda}\ref{it:lambda subcritical no assumptions}.

\begin{lemma}[Weight vector for the subcritical regime]\label{lemma:weight_subcritical}
    For each $0 < \varepsilon \leq 1$ and every non-empty graph $G$, there exists a vector $w \in [0,1]^{V(G)}$ such that for each $x \in V(G)$ we have
    \begin{enumerate}
        \item $w(x) \geq \frac{\varepsilon\lambda(G)}{2\sqrt{\ab G \Delta(G)}}$, \text{and}
        \item $\sum_{y \in N(x)} w(y) \leq (1+\varepsilon)\lambda(G)w(x)$.
    \end{enumerate}
\end{lemma}
\begin{proof}
    We write $\lambda = \lambda(G)$ and for each $x \in V(G)$ and $k\geq 0$, let $W_x^{(k)}$ denote the number of walks in $G$ of length $k$ that start at $x$.
    We define
    \[
        u(x) = \sum_{k=0}^\infty \left((1+\varepsilon)\lambda\right)^{-k}W_x^{(k)}.
    \]
    We will show that $u(x) \leq \frac{2\sqrt{\ab G \Delta(G)}}{\varepsilon\lambda}$.
    Then taking $w(x) = u(x)\cdot \frac{\varepsilon\lambda}{2\sqrt{\ab G \Delta(G)}}$ and using $u(x)\geq 1$ ensures that the first property is satisfied while also $w(x)\leq 1$.
    Moreover, since
    \[
        u(x) = 1 + \bigl((1+\varepsilon)\lambda\bigr)^{-1} \cdot \sum_{y \in N(x)} u(y) \geq \bigl((1+\varepsilon)\lambda\bigr)^{-1} \cdot \sum_{y \in N(x)} u(y),
    \]
    the second property is satisfied as well.
    
    To give an upper bound on $u(x)$, we let $\mathds{1}$ be the all-ones vector and $e_x$ be the vector with $e_x(x) = 1$ and all other entries set to $0$.
    Then note that for each $k \geq 1$ we have
    \[
        W_x^{(k)} = e_x^TA_G^k \mathds{1} \leq \lVert A_Ge_x \rVert_2\lVert A_G^{k-1}\mathds{1}\rVert_2 \leq \sqrt{d(x)} \cdot \lambda^{k-1} \sqrt{|G|} \leq \lambda^{k-1}\sqrt{\ab G \Delta(G)},
    \]
    where we used the Cauchy--Schwarz inequality and that $\lambda$ is the operator norm of $A_G$.
    In particular, we get that for each $x \in V(G)$,
    \[
        u(x) =\sum_{k=0}^\infty \bigl((1+\varepsilon)\lambda\bigr)^{-k}W_x^{(k)} \leq 1  + \frac{\sqrt{\ab G \Delta(G)}}{\lambda} \sum_{k=1}^{\infty} (1+\varepsilon)^{-k} \leq \frac{2\sqrt{\ab G \Delta(G)}}{\varepsilon\lambda},
    \]
    which concludes the proof.
\end{proof}

\subsection{Concentration inequality for the exploration process}
In this subsection, we want to prove the following lemma, which says that it is unlikely that we see much more (or much less) weight on the ``yeses'' during the exploration process than we would expect.
In our applications of this lemma, $q_i$ will be the weight of the endpoint of the edge considered at step $i$ that lies outside of the currently revealed component.
Although $q_i$ may depend on the outcomes of the previous queries, it is determined before the outcome of the $i$th query is revealed.
\begin{lemma}[Concentration inequality for exploration]\label{lemma:tail_bounds}
    Let $m \in \mathbb{N}$, $0 \leq a \leq 1/2$ and $p \in [0,1]$.
    Let $Y_1, \dots, Y_m$ be mutually independent $\mathrm{Ber}(p)$ random variables.
    For each $t \in [m]$, write $\mathcal{F}_{t} = \sigma(Y_1, \dots, Y_t)$ and let $q_t$ be an $\mathcal{F}_{t-1}$-measurable random variable taking values in $[0,1]$.
    Then for each $C \geq 0$ we have
    \begin{enumerate}
        \item\label{it:bound more yes} $\Pr[\sum_{i=1}^m \left(aY_iq_i - pa(1+a)q_i\right) \geq C] \leq e^{-C}$, \text{and}
        \item\label{it:bound more no} $\Pr[\sum_{i=1}^m \left( aY_iq_i - pa(1-a)q_i \right) \leq -C] \leq e^{-C}$.
    \end{enumerate}
\end{lemma}
\begin{proof}
    To prove the lemma, we show that a suitably defined sequence of variables is a supermartingale and then apply Markov's inequality.
    The proofs of both parts follow the same lines.

    \vspace{5px}
    \noindent\textbf{Part \ref{it:bound more yes}.} Let $M_0 = 1$ and for each $t \in [m]$ define
    \[
        M_t = \exp\left(\sum_{i=1}^t \bigl(aY_iq_i - pa(1+a)q_i\bigr)\right).
    \]
    We note that for any $t \in [m]$ we have
    \[
        \mathbb{E}\Bigl[\exp(aY_tq_t) \mid \mathcal{F}_{t-1}\Bigr] = 1-p +p\exp(aq_t) \leq 1+pa(1+a)q_t \leq \exp\bigl(pa(1+a)q_t\bigr),
    \]
    where we used that $\exp(aq) \leq 1+aq +a^2q^2 \leq 1+a(1+a)q$ for any $q \in [0,1]$ and $a\in [0,1/2]$.
    In particular, we get that
    \[
        \mathbb{E}\left[M_t \mid \mathcal{F}_{t-1}\right] = M_{t-1} \cdot \exp\bigl(-pa(1+a)q_t \bigr)\cdot \mathbb{E}\bigl[\exp(aY_tq_t) \mid \mathcal{F}_{t-1}\bigr] \leq M_{t-1}.
    \]
    Hence, $\mathbb{E}[M_m] \leq \mathbb{E}[M_0] =1$ and so by Markov's inequality, for any $C \geq 0$, we get
    \[
        \Pr\left[\sum_{i=1}^m \left(aY_iq_i - pa(1+a)q_i\right) \geq C\right] = \Pr\left[M_m \geq e^C\right]\leq e^{-C}.
    \]

    \vspace{5px}
    \noindent\textbf{Part \ref{it:bound more no}.} The proof is very similar. Write $N_0 = 1$ and for each $t \in [m]$ define
    \[
        N_t = \exp\left(\sum_{i=1}^t  \bigl(pa(1-a)q_i - aY_iq_i  \bigr)\right).
    \]
    We notice that for any $t \in [m]$ we have
    \[
        \mathbb{E}\Bigl[\exp(-aY_tq_t) \mid \mathcal{F}_{t-1}\Bigr] = 1-p+p\exp(-aq_t) \leq 1 -pa(1-a)q_t \leq \exp\bigl(-pa(1-a)q_t\bigr),
    \]
    where we used that $\exp(-aq) \leq 1-aq + a^2q^2 \leq 1 -a(1-a)q$ for any $q \in [0,1]$ and $a \in [0,1/2]$.
    In particular,
    \[
        \mathbb{E}\bigl[N_t \mid \mathcal{F}_{t-1}\bigr] = N_{t-1}\cdot\exp\bigl(pa(1-a)q_t\bigr) \cdot  \mathbb{E}\bigl[\exp(-aY_tq_t) \mid \mathcal{F}_{t-1}\bigr] \leq N_{t-1}.
    \]
    Hence, $\mathbb{E}[N_m] \leq \mathbb{E}[N_0] = 1$ and so again by Markov's inequality for any $C \geq 0$ we get
    \[
        \Pr\left[\sum_{i=1}^m \bigl(aY_iq_i - pa(1-a)q_i\bigr) \leq -C \right] = \Pr[N_m \geq e^C] \leq e^{-C},
    \]
    as desired.
\end{proof}

\section{Exploring the connected component of a vertex}\label{section:component_of_a_vertex}
In this section, we want to understand the behaviour of the weight and size of the connected component $\mathcal{C}_v$ of $G_p$ containing a fixed vertex $v$.
We fix a weight vector $w \in [0,1]^{V(G)}$, which in our applications will either be the eigenvector corresponding to $\lambda(G)$ or the vector given by \cref{lemma:weight_subcritical}.
Then, in the subcritical regime, we can show that the probability that $w(\mathcal{C}_v) \geq C$ is exponentially small in $C$.
Similarly, for the supercritical regime the probability that $w(\mathcal{C}_v) \geq C$ but $|\mathcal{C}_v| \leq \gamma\lambda(G)$ is also exponentially small in $C$.
Finally, for the proof of \cref{thm:main avg degree}, we also show that if $\mathcal{C}_v$ is small, then it a.a.s.\ does not touch too many edges in $G$.
That is, removing $\mathcal{C}_v$ from $G$ will not decrease the average degree significantly.

The three estimates are collected in the following lemma.
\begin{lemma}[Component of a fixed vertex]\label{lemma:component_of_a_vertex}
    For each $0 < \varepsilon \leq 1$ there exists $\gamma > 0$ such that the following holds.
    Let $G$ be a graph, let $\lambda >0$ and let $w \in [0,1]^{V(G)}$.
    Let $0 < p \leq 1$, fix a vertex $v \in V(G)$ and let $\mathcal{C}_v$ denote the connected component of $G_p$ containing $v$.
    Then the following hold.
    \begin{enumerate}[label=(\alph*)]
        \item\label{it:subcritical sharpness} Suppose that $p \leq (1-\varepsilon)/\lambda$ and that $\sum_{y \in N(x)} w(y) \leq \lambda w(x)$ for all $x \in V(G)$.
        Then for all $C \geq 0$ we have
        \[
            \Pr\bigl[w(\mathcal{C}_v) \geq C\bigr] \leq \exp\left(1-\frac{\varepsilon^2C}{4}\right).
        \]
        \item\label{it:supercritical sharpness} Suppose that $p \geq (1+\varepsilon)/\lambda$ and that $\sum_{y \in N(x)} w(y) \geq \lambda w(x)$ for all $x \in V(G)$.
        Then for all $C \geq 0$ we have
        \[
            \Pr\Bigl[|\mathcal{C}_v| \leq \gamma\lambda \text{ and } w(\mathcal{C}_v) \geq C\Bigr] \leq \exp\left(-\frac{\varepsilon^2C}{18}\right).
        \]
        \item\label{it:not too many edges} For all $k \geq 6$ we have
        \[
            \Pr\left[|\mathcal{C}_v| \leq p^{-1} \text{ and } \left(\sum_{x \in \mathcal{C}_v} d_G(x)\right) \geq kp^{-2}\right] \leq \exp\left(-\frac{k}{8p}\right).
        \]
    \end{enumerate}
\end{lemma}
%\noindent We note that the second case of the lemma in particular implies that for a vertex $v$ with $w(v)=1$, the probability that $|\mathcal{C}_v| \geq \gamma\lambda$ is at least $1-\exp(-\varepsilon^2/18)$.\yuval{Does this sentence make sense here? It's good to include somewhere but I'm not sure here is best}

To prove \cref{lemma:component_of_a_vertex} we will explore $\mathcal{C}_v$ by revealing one edge of $G_p$ at a time.
More formally, we will analyse the following random process.
\begin{definition}[Exploration process]\label{definition:exploration_process}
    Let $G$ be a graph, fix $v \in V(G)$ and let $\prec$ be a total ordering of $E(G)$.
    Let $p \in [0, 1]$, write $m = |E(G)|$ and let $Y_1, \dots, Y_m$ be independent $\mathrm{Ber}(p)$ random variables.
    We let $C_0 = \{ v \}$ and initially mark all edges as unexposed.
    In each step $1 \leq i \leq m$, if there is no unexposed edge in $E(C_{i-1}, V(G) \setminus C_{i-1})$, we set $C_i=C_{i-1}$.
    Otherwise, we pick the $\prec$-smallest unexposed edge $x_iy_i \in E(C_{i-1}, V(G) \setminus C_{i-1})$ with $y_i \notin C_{i-1}$.
    If $Y_i = 1$, we set $C_i = C_{i-1} \cup \{y_i\}$, otherwise we set $C_i = C_{i-1}$.
    We then mark $x_iy_i$ as exposed and proceed to the next step.
\end{definition}
\noindent This exploration process is at the heart of our argument and, crucially, has the following two properties.
\begin{observation}
    In the exploration process, all edges in $E(C_m, V(G) \setminus C_m)$ are exposed, while no edges in $E(V(G)\setminus C_m)$ are exposed. 

    Additionally, the final set $C_m$ has the same distribution as $\mathcal C_v$, the component of $G_p$ containing \nolinebreak $v$.
    % \yuval{I rewrote this as an official observation and added a short (slightly sketchy) proof}
\end{observation}
\begin{proof}
    The first statement follows immediately from the definition of the exploration process, as an edge is only exposed if, at some point $i$ in the process, it has exactly one endpoint in $C_{i-1}$ and one endpoint outside of $C_{i-1}$. Moreover, for the exploration process to terminate, all edges leaving $C_m$ must have been queried. 

    For the second statement, we couple $G_p$ and the exploration process as follows. Let $\{X_e:e \in E(G)\}$ be the indicator random variables for the event that $e \in E(G_p)$. For each $i \in [m]$, where $m=\ab{E(G)}$, let $Y_i=X_{x_i y_i}$ if $x_i y_i$ is the edge explored at step $i$, and otherwise let $Y_i$ be an independent Ber$(p)$ random variable. As $C_m$ is revealed to be the component of $v$ in the exploration process, we see that under this coupling, $C_m=\mathcal C_v$.
    % For the second statement, once the exploration process has ended, let us keep each unexposed edge independently with probability $p$. By doing so, we simply sample $G_p$. By the first statement, the exploration process has already revealed that $C_m$ spans a connected subgraph of $G_p$, while also revealing that there is no edge of $G_p$ joining $C_m$ to $V(G) \setminus C_m$, hence $C_m$ is precisely the component of $v$ in this outcome of $G_p$, proving the second statement.
\end{proof}
% First, since each edge of $G$ queried throughout the process is considered present with probability $p$ independently of all other edges, the final set $C_m$ is distributed as $\mathcal{C}_v$, the component of $G_p$ containing $v$. Second, the exploration process exposes all edges in $E(C_{m}, V(G) \setminus C_{m})$ and some edges in $E(C_m)$ while no edges in $E(V(G)\setminus C_{m})$ get exposed. \yuval{Should we dwell more on this point? It's basically obvious but a bit difficult to formally justify, and it's crucial for the remainder of the argument. \textbf{MC:} A added a little bit more, if u have more to say it woudl certainly be good to add.}

As discussed before, in the proof of \cref{lemma:component_of_a_vertex} we want to argue that on the respective events we see much more (or much less) weight on the ``noes" than we would expect, which by \cref{lemma:tail_bounds} is very unlikely.

\begin{proof}[Proof of \cref{lemma:component_of_a_vertex}]
    Let $m = |E(G)|$, pick an arbitrary ordering $\prec$ of $E(G)$ and let $Y_1, \dots, Y_m$ and $C_1, \dots, C_m$ be defined as in \cref{definition:exploration_process}.
    For each $i \in [m]$, we let $q_i = w(y_i)$ if at step $i$ we have considered an unexposed edge $x_iy_i$ and we set $q_i = 0$ otherwise.
    Note that $q_i$ is a $\sigma(Y_1, \dots, Y_{i-1})$-measurable random variable that takes values in $[0,1]$.
    Additionally, $\mathcal{C}_v$ and $C_m$ are identically distributed, so it is enough to consider $C_m$.
    The proofs of the three parts follow similar lines.
    
    \vspace{5px}
    \noindent{\textbf{Part \ref{it:subcritical sharpness}}}
    We want to argue that on the event that $w(\mathcal{C}_v) \geq C$ we have seen much more weight on the ``yeses'' than expected, which is very unlikely by \cref{lemma:tail_bounds}.
    To that end, let $a = \varepsilon/2$ and notice that all the edges we have exposed during the exploration process have at least one endpoint in $C_m$.
    In particular, the total weight queried is $\sum_{i = 1}^m q_i \leq \sum_{x \in C_m}\sum_{y \in N(x)} w(y)$.
    On the other hand, the weight of the ``yeses'' is $\sum_{i=1}^m Y_iq_i = w(C_m) - w(v) \geq w(C_m) - 1$.
    Therefore,
    \begin{align*}
        \sum_{i=1}^m\left( aY_iq_i - pa(1+a)q_i\right) &\geq a\bigl(w(C_m) - 1\bigr) - pa(1+a)\Big(\sum_{x \in C_m}\sum_{y \in N(x)} w(y)\Bigr)\\
        &\geq \frac{\varepsilon}{2}\Bigl(w(C_m)-1\Bigr)-\frac{\varepsilon}{2}\Bigl(1+\frac{\varepsilon}{2}\Bigr)\Bigl(1-\varepsilon\Bigr)w(C_m)\\
        &\geq\frac{\varepsilon^2w(C_m)}{4}  - 1.
    \end{align*}
    Hence, by \cref{lemma:tail_bounds} we get that
    % \yuval{GPT complains that the final step is unjustified if $\varepsilon^2C/4-1$ is negative, since \cref{lemma:tail_bounds} requires $C \geq 0$. But the bound is trivial in that case so I'm not sure we care}
    \[
        \Pr\Bigl[w(C_m) \geq C\Bigr] \leq \Pr\left[ \sum_{i=1}^m\bigl( aY_iq_i - pa(1+a)q_i\bigr) \geq \frac{\varepsilon^2C}{4} - 1\right] \leq \exp\left(1 - \frac{\varepsilon^2C}{4}\right).
    \]
    
    \vspace{5px}
    \noindent{\textbf{Part \ref{it:supercritical sharpness}}}
    Here, we in turn want to argue that on the event $|\mathcal{C}_v|\leq \gamma\lambda$ and $w(\mathcal{C}_v) \geq C$, the total weight of the ``yeses'' during the exploration must have been less than we expected.
    To that end, set $\gamma = \varepsilon/6$ and notice that during the exploration process we have exposed all the edges in $E(C_m, V(G) \setminus C_m)$.
    In particular, by \cref{lemma:small_sets_expand}, on the event that $|C_m| \leq \gamma\lambda$ the total weight explored during the process is
    \begin{align*}
        \sum_{i=1}^m q_i \geq \sum_{x \in C_m}\sum_{y \in N(x) \setminus C_m } w(y) 
        %\\
        %&= \sum_{x \in C_m} \left( \sum_{y \in N(x)} w(y)-\sum_{y \in N(x)\cap C_m} w(y)
        %\right)\\
        %&\geq \lambda w(C_m)-\ab{C_m}w(C_m)\\
        \geq \left(1-\gamma\right)\lambda w(C_m).
    \end{align*}
    On the other hand, the weight of the ``yeses'' is $\sum_{i=1}^m Y_i q_i = w(C_m) - w(v) \leq w(C_m)$.
    In particular, on the event that $w(C_m) \geq C$ and $|C_m| \leq \gamma \lambda$ we have
    \[
        \sum_{i = 1}^m\bigl( \gamma Y_iq_i - p\gamma(1-\gamma)q_i\bigr) \leq \gamma\Bigl(1- (1+\varepsilon)\bigl(1-\gamma\bigr)^2\Bigr)w(C_m) \leq -\frac{\varepsilon^2C}{18}
    \]
    Then, by \cref{lemma:tail_bounds} with $a=\gamma$, we get
    \[
        \Pr\left[ |C_m| \leq\gamma\lambda\text{ and }w(C_m) \geq C \right] \leq \Pr\Bigl[\sum_{i = 1}^m\bigl( \gamma Y_iq_i - p\gamma(1-\gamma)q_i \bigr) \leq -\frac{\varepsilon^2C}{18}\Bigr] \leq \exp\left(-\frac{\varepsilon^2C}{18}\right),
    \]
    which proves the statement.

    \vspace{5px}
    \noindent{\textbf{Part \ref{it:not too many edges}}} In this part, we want to make a similar argument to Part \ref{it:supercritical sharpness} and use that in our event we must have seen many fewer  ``yeses'' than expected.
    To capture the size instead of the weight of $C_m$, for each $i \in [m]$ let $q_i' = 1$ if an unexposed edge was considered at the $i$th step of the exploration process and set $q_i' = 0$ otherwise.
    We note that $q_i'$ is $\sigma(Y_1, \dots, Y_{i-1})$-measurable and takes values in $[0,1]$.
    
    Now, recall that during the exploration process we must have queried all edges in $E(C_m, V(G)\setminus C_m)$.
    In particular, on the event that $|C_m| \leq p^{-1}$ and $\sum_{x \in C_m} d_G(x) \geq kp^{-2}$, the total number of queries during the exploration is
    \[
        \sum_{i=1}^m q_i' \geq \ab{E(C_m, V(G) \setminus C_m)}\geq  \sum_{x \in C_m}\bigl(d_G(x) - |C_m|\bigr) = \left(\sum_{x \in C_m} d_G(x) \right) - |C_m|^2 \geq (k-1)p^{-2}.
    \]
    On the other hand, the number of successful queries is $\sum_{i=1}^mY_iq_i' \leq |C_m|\leq p^{-1}$.
    Therefore,
    \[
        \sum_{i = 1}^m \left(\frac{Y_iq_i'}{2} -\frac{pq_i'}{4}\right) \leq \frac{1}{2p} - \frac{k-1}{4p} \leq -\frac{k}{8p},
    \]
    where we use $k\geq 6$ in the last inequality.
    By applying \cref{lemma:tail_bounds} with $a=1/2$, we get
    \[
        \Pr\left[|\mathcal{C}_v| \leq p^{-1} \text{ and } \left(\sum_{x \in \mathcal{C}_v} d_G(x)\right) \geq kp^{-2}\right] \leq \Pr\left[\sum_{i = 1}^m \left( \frac{Y_iq_i'}{2} - \frac{pq_i'}{4}\right) \leq -\frac{k}{8p}\right] \leq \exp\left(-\frac{k}{8p}\right),
    \]
    which concludes the proof of the lemma.
\end{proof}

\section{Proofs of the main results}\label{section:proofs of theorems}
In this section, we use the results from \cref{section:component_of_a_vertex} to deduce \cref{thm:main avg degree} and \cref{thm:main large lambda}.
Then, we will prove \cref{thm:main lambda} by combining \cref{thm:main large lambda} with a separate argument for the case when $\lambda(G) = O(\sqrt{\Delta(G)})$. To avoid confusion between random and deterministic objects, we often use bold letters to denote random objects in this section.

\subsection{Proof of Theorem~\ref{thm:main large lambda}}
We begin with the proof of \cref{thm:main large lambda}, starting with the supercritical regime.
That is, we have a graph $G$ with $\Delta(G) \leq \lambda(G)^2/K^2$ and want to prove that with probability $1-\exp(-cK)$, $G_p$ contains a component of size $\Omega(\lambda(G))$ if $p \geq (1+\varepsilon)/\lambda(G)$.
For that, we let $w \in [0,1]^{V(G)}$ be the normalized eigenvector corresponding to $\lambda(G)$ such that $w(v)=1$ for some $v \in V(G)$.
For this $v$, by \cref{lemma:component_of_a_vertex} we know that with constant probability the connected component of $G_p$ containing $v$ has size at least $\gamma\lambda$ and that if not, then a.a.s.\ $w(\mathcal{C}_v)$ is small.
In the latter case, we would like to repeat the argument inside $G - \mathcal{C}_v$ and argue that for some other $v'$ we can find a large connected component with constant probability.
For that, we use that $\Delta(G) \leq \lambda(G)^2/K^2$, and so by  \cref{lemma:spectral_radius_stability} $\lambda(G-\mathcal{C}_v)$ won't decrease much compared to $\lambda(G)$.
We are able to iterate this argument $\Omega(K)$ times, so the probability of not having found a large component throughout is $\exp(-\Omega(K))$.

\begin{proof}[Proof of Theorem~\ref{thm:main large lambda}\ref{it:large lambda supercritical}]
    Given $\varepsilon>0$, let $\gamma'$ be as defined by \cref{lemma:component_of_a_vertex} applied with $\varepsilon/2$, and set $\gamma = \gamma'/2$. Let $G$ be a graph, $p\geq (1+\varepsilon)/\lambda(G)$ and let $K$ be such that $\Delta(G)\leq\lambda(G)^2/K^2$. Let $\mathcal{E}^{\mathrm{nogiant}}$ be the event that $G_p$ contains no connected component of size at least $\gamma\lambda(G)$.
    We want to show that
    \[
        \Pr[\mathcal{E}^{\mathrm{nogiant}}] \leq \exp(-cK),
    \]
    where $c$ is a constant depending only on $\varepsilon$.

    We first argue that we may assume that $K$ is sufficiently large compared to $\varepsilon$.
    To that end, let $w \in [0,1]^{V(G)}$ be an eigenvector of $G$ corresponding to $\lambda(G)$ such that $w(v) = 1$ for some $v \in V(G)$.
    Let $\mathbf{\mathcal{C}}_v$ denote the connected component of $G_p$ containing $v$.
    Then, by \cref{lemma:component_of_a_vertex}\ref{it:supercritical sharpness} applied with $C=1$, noticing that with probability $1$ we have $w(\mathbf{\mathcal{C}}_v) \geq 1$, we immediately get that $\Pr[|\mathbf{\mathcal{C}}_v| \leq \gamma\lambda(G)] \leq \exp(-\varepsilon^2/18) \leq \exp(-cK)$, whenever $c \leq \varepsilon^2/(18K)$. In particular, by choosing $c$ sufficiently small with respect to $\varepsilon$, we may assume that $K$ is sufficiently large.

    So we assume that $K$ is large enough and consider the following random process, which we initialize with $\mathbf{S}_1 = \emptyset$.
    We let $T = \sqrt{\varepsilon}K/8$, which will be the maximum number of steps we make.
    In each step $t \in [T]$, we let $\mathbf{w}_t \in [0,1]^{V(G) \setminus \mathbf{S}_t}$ be the eigenvector of $G-\mathbf{S}_t$ corresponding to the eigenvalue $\lambda(G - \mathbf{S}_t)$ such that for some $\mathbf{v}_t \in V(G) \setminus \mathbf{S}_t$ we have $\mathbf{w}_t(\mathbf{v}_t) =1$.
    We then reveal the connected component $\mathbf{C}_t$ of $G_p$ containing $\mathbf{v}_t$.
    If $|\mathbf{C}_t| \geq \gamma\lambda(G)$, we terminate the process with success.
    Otherwise, if $\mathbf{w}_t(\mathbf{C}_t) \geq \sqrt{\varepsilon}K/8$, we terminate the process with failure.
    Finally, if neither of the two holds, then we let $\mathbf{S}_{t+1} = \mathbf{S}_t \cup \mathbf{C}_t$ and continue.

    To analyse the process, for each $t \in \{0,1, \dots, T\}$ we let $\mathcal{E}_t^{\mathrm{nowin}}$ and $\mathcal{E}_t^{\mathrm{nofail}}$ denote the events that the process has not terminated with success or failure, respectively, up to and including step $t$.
    We note that $\mathcal{E}_t^{\mathrm{nowin}}, \mathcal{E}_t^{\mathrm{nofail}} \in \sigma(\mathbf{C}_1, \dots, \mathbf{C}_t)$ and that we have $\Pr[\mathcal{E}^{\mathrm{nogiant}}] \leq \Pr[\mathcal{E}_T^{\mathrm{nowin}}]$. Therefore, it suffices to upper-bound $\Pr[\mathcal{E}_T^{\mathrm{nowin}}]$.

    We first want to understand the probability of success or failure at step $t+1$ given that the process has not terminated in the first $t$ steps.
    For that, we fix $t \in \{0,1, \dots, T-1\}$, let $C_1, \dots, C_t \subseteq  V(G)$ and let $\mathcal E$ be the event that $(\mathbf{C}_1, \dots, \mathbf{C}_t) = (C_1, \dots, C_t)$.
    We suppose that $\mathcal E$ holds with positive probability and that $\mathcal E \subseteq \mathcal{E}_t^{\mathrm{nowin}} \cap \mathcal{E}_t^{\mathrm{nofail}}$.
    For each $1 \leq t' \leq t+1$ we let $S_{t'} = \bigcup_{i=1}^{t'-1} C_i$ and let $w_{t'} \in [0,1]^{V(G) \setminus S_{t'}}$ be the eigenvector of $G - S_{t'}$ corresponding to $\lambda(G- S_{t'})$ with $\lVert w_{t'}\rVert_{\infty} =1$.

    We now argue that the eigenvalue $\lambda(G -S_{t+1})$ has not decreased significantly compared to $\lambda(G)$.
    More specifically, we claim that for each $1 \leq t' \leq t+1$ we have $\lambda(G-S_{t'}) \geq (1-(t'-1)\sqrt{\varepsilon}/K)\lambda(G)$, which is certainly true for $t' =1$.
    Assume therefore that this holds for some $t' \leq t$.
    We then have $\lambda(G - S_{t'}) \geq \lambda(G)/2 \geq (K/2)\sqrt{\Delta(G-S_{t'})}$ and $w_{t'}(C_{t'}) < \sqrt{\varepsilon}K/8$, since $\mathcal E \subseteq \mathcal{E}_t^{\mathrm{nowin}} \cap \mathcal{E}_t^{\mathrm{nofail}}$.
    Therefore, by \cref{lemma:spectral_radius_stability} we get
    \[
        \lambda(G - S_{t'+1}) \geq \left( 1- \frac{2w_{t'}(C_{t'})}{K^2/4} \right) \lambda(G-S_{t'}) \geq \left( 1-\frac{\sqrt \varepsilon}{K}\right) \lambda(G-S_{t'})\geq\left( 1-\frac{t'\sqrt{\varepsilon}}{K} \right)\lambda(G),
    \]
    where the second inequality uses that $\mathcal E \subseteq \mathcal E_t^{\mathrm{nofail}}$ and the final inequality uses that the claimed bound holds for $t'$.

    We can now apply \cref{lemma:component_of_a_vertex} to bound the probability of success and of failure at step $t+1$.
    To that end, we first notice that conditioned on the event $\mathcal E$ the graph $G_p - S_{t+1}$ is distributed as $(G - S_{t+1})_p$.
    Moreover, by the above considerations we have $p \geq (1+\varepsilon/2)/\lambda(G-S_{t+1})$ and $\lambda(G-S_{t+1}) \geq \lambda(G)/2$.
    Since $w_{t+1}(\mathbf{C}_{t+1}) \geq w_{t+1}(\mathbf{v}_{t+1}) = 1$ with probability $1$, by part \ref{it:supercritical sharpness} of \cref{lemma:component_of_a_vertex} applied with $\varepsilon/2$ and $C = 1$, we get
    \[
        \Pr\left[ \mathcal{E}_{t+1}^{\mathrm{nowin}} \mid \mathcal E\right] \leq \Pr\left[|\mathbf{C}_{t+1}| \leq 2\gamma\lambda(G-S_{t+1})\mid \mathcal E\right] \leq \exp\left( -\frac{\varepsilon^2}{72}\right).
    \]
    Similarly, again by part \ref{it:supercritical sharpness} of \cref{lemma:component_of_a_vertex}, but with $C = \sqrt{\varepsilon}K/8$, we get
    \[
        \Pr\left[ \neg \mathcal{E}_{t+1}^{\mathrm{nofail}} \mid \mathcal E\right] \leq \Pr\left[|\mathbf{C}_{t+1}| \leq 2\gamma\lambda(G-S_{t+1}) \text{ and } w_{t+1}(\mathbf{C}_{t+1}) \geq \frac{\sqrt{\varepsilon}K}{8} \mid \mathcal E\right] \leq \exp \left( - \frac{\varepsilon^{5/2}K}{576} \right).
    \]
    Finally, we notice that since we can partition $\mathcal{E}_t^{\mathrm{nowin}} \cap \mathcal{E}_t^{\mathrm{nofail}}$ by such events $\mathcal E$, we get the same bounds on $\Pr[\mathcal{E}_{t+1}^{\mathrm{nowin}}\mid \mathcal{E}_t^{\mathrm{nowin}} \cap \mathcal{E}_t^{\mathrm{nofail}}]$ and $\Pr[\neg\mathcal{E}_{t+1}^{\mathrm{nofail}}\mid \mathcal{E}_t^{\mathrm{nowin}} \cap \mathcal{E}_t^{\mathrm{nofail}}]$ respectively.

    To conclude the proof, we note that for the event $\mathcal{E}_T^{\mathrm{nowin}}$ to hold, the process must have either terminated with failure at some step $t<T$ or run all the way up until the end without ever terminating with success.
    Therefore,
    \begin{align*}
        \Pr[\mathcal{E}_T^{\mathrm{nowin}}] &\leq \left( \prod_{t=0}^{T-1}  \Pr\left[\mathcal{E}_{t+1}^{\mathrm{nowin}}\mid \mathcal{E}_t^{\mathrm{nowin}} \cap \mathcal{E}_t^{\mathrm{nofail}}\right]\right)+\left(\sum_{t=0}^{T-2} \Pr\left[\neg\mathcal{E}_{t+1}^{\mathrm{nofail}} \mid \mathcal{E}_t^{\mathrm{nowin}} \cap \mathcal{E}_t^{\mathrm{nofail}}\right] \right) \\
        &\leq \exp \left( -\frac{\varepsilon^2}{72}\cdot \frac{\sqrt{\varepsilon}K}{8} \right) +\frac{\sqrt{\varepsilon}K}{8}\cdot \exp\left( - \frac{\varepsilon^{5/2}K}{576} \right)\\
        &\leq \exp\left( -cK \right),
    \end{align*}
    for some constant $c$ depending on $\varepsilon$, where we used that $K$ is sufficiently large compared to $\varepsilon$.
    % \yuval{GPT points out that this is still wrong, e.g.\ if $\varepsilon=1$ and $K=9$ then the penultimate quantity is $>1$}
    % \yuval{GPT points out that the final step is only true if $K$ is sufficiently large wrt $\varepsilon$. But if it's not then I think that we only need to apply \cref{lemma:component_of_a_vertex}\ref{it:subcritical sharpness} once, and therefore the second term is irrelevant and things are again OK. So I'm not really sure anything needs to be changed but perhaps we should mention addressing small $K$?}
\end{proof}

Turning our attention to the subcritical phase, we would like to use \cref{lemma:component_of_a_vertex} to argue the probability that the component of a fixed vertex is large is very small and then union bound over all vertices of our graph.
Naively, we would take $w$ to again be the eigenvector of $G$ corresponding to $\lambda(G)$ and immediately get that  a.a.s.\ there is no connected component $\mathcal C$ in $G_p$ with $w(\mathcal C) = \Omega( \log {|G|})$.
As discussed earlier, this however doesn't imply that there is no connected component of large size.
To get around this issue, we will instead take $w$ to be the vector from \cref{lemma:weight_subcritical}, which will allow us to control the size of any set in terms of its weight.

\begin{proof}[Proof of Theorem~\ref{thm:main large lambda}\ref{it:large lambda subcritical}]
    Let $0 < \varepsilon \leq 1$, let $G$ be a graph and fix $p \leq (1-\varepsilon)/\lambda(G)$.
    Let $L = 500/\varepsilon^3$ and $\lambda = (1+\varepsilon/4)\lambda(G)$.
    By \cref{lemma:weight_subcritical} applied with parameter $\varepsilon/4$ there exists a vector $w \in [0,1]^{V(G)}$ such that for each $x \in V(G)$, we have 
    \[
        w(x) \geq \frac{\varepsilon\lambda(G)}{8 \sqrt{\ab G \Delta(G)}} \hspace{10px} \text{and} \hspace{10px} \sum_{y \in N(x)} w(y) \leq \lambda w(x).
    \]

    We first want to bound the probability that the component of a fixed vertex $v\in V(G)$ is large and later apply the union bound over all $v \in V(G)$.
    To that end, let $\mathcal C_v$ denote the connected component of $G_p$ containing $v$.
    Notice that $p \leq (1-\varepsilon/2)/\lambda$ and that for any $S \subseteq V(G)$ we have $w(S) \geq |S|\cdot \frac{\varepsilon \lambda(G)}{8 \sqrt{\ab{G} \Delta(G)}} \geq |S|\frac{\varepsilon K}{8\sqrt{\ab{G}}}$.
    Therefore, by \cref{lemma:component_of_a_vertex}\ref{it:subcritical sharpness}, we get that
    \[
        \Pr\left[|\mathcal C_v| \geq L\frac{\sqrt{\ab{G}} \log {\ab{G}}}{K}\right] \leq \Pr\left[w(\mathcal C_v) \geq  \frac{\varepsilon L \log {\ab{G}} }{8}\right] \leq \exp \left(1-\frac{\varepsilon^3 L \log{\ab G}}{128}\right)\leq \ab{G}^{-3},
    \]
    where we used our choice of $L$.
    Hence, by union bound over all $v \in V(G)$ we get
    \[
        \Pr\left[\text{$G_p$ has a component of order at least $L\frac{\sqrt{\ab{G}} \log {\ab{G}}}{K}$}\right] \leq \ab{G} \cdot \ab{G}^{-3} = o(1)
    \]
    as $\ab{G} \to \infty$.
\end{proof}

\subsection{Proof of Theorem~\ref{thm:main avg degree}}
We now want to prove \cref{thm:main avg degree}, i.e., if a graph $G$ has average degree $d$ and $p \geq (1+\varepsilon)/d$, then a.a.s.\ $G_p$ contains a component of size $\Omega(d)$.
We split the proof into two regimes.
Firstly, if $d \geq \ab{G}^{3/4}$, then $\lambda(G) \geq d \geq \ab{G}^{1/4}\sqrt{\Delta(G)}$, as clearly $\Delta(G) \leq \ab{G}$.
Therefore, by \cref{thm:main large lambda} we immediately get that a.a.s.\ $G_p$ contains a component of size $\Omega(\lambda(G)) = \Omega(d)$.

What remains is thus the case when $d \leq \ab{G}^{3/4}$.
Here, we do a similar argument as in the proof of \cref{thm:main large lambda}\ref{it:large lambda supercritical}.
That is, we repeatedly reveal a connected component of a vertex and remove all of it from the graph, until we (hopefully) find one which has size $\Omega(d)$.
While it is no longer true that throughout the process the spectral radius of the graph does not decrease significantly, we are able to argue that the average degree a.a.s.\ does not. Towards this, we show that every small component a.a.s.\ touches only a few edges.
% \yuval{GPT points out that this isn't really true, it touches $Td^2$ edges for some $T$ depending on $n,d$. Do we care? \textbf{MC: }I changed it, see if you prefer}

%at most $O(d^2)$ edges and combine this observation with $d \leq \ab G^{3/4}$.

\begin{proof}[Proof of Theorem~\ref{thm:main avg degree}]
    Fix $0 < \varepsilon \leq 1$, let $G$ be a graph with average degree $d$ and let $n =|G|$ be large enough such that $\sqrt{\varepsilon n^{1/4}/16} \geq 6$.
    Let\footnote{For any $p\leq p'$, we may couple $G_p$ and $G_{p'}$ so that $G_p \subseteq G_{p'}$ with probability $1$, hence it suffices to prove the theorem for $p=(1+\varepsilon)/d$.} $p = (1+\varepsilon)/d$ and let $\gamma'$ be the constant from \cref{lemma:component_of_a_vertex} with parameter $\varepsilon/2$.
    Finally, let $\gamma = \gamma'/2$ and let $\mathcal{E}^{\mathrm{nogiant}}$ denote the event that $G_p$ contains no component of size at least $\gamma d$.
    We will show that
    \[
        \Pr[\mathcal{E}^{\mathrm{nogiant}}] \leq \exp\bigl(-c_\varepsilon n^{1/8} \bigr),
    \]
    for some constant $c_\varepsilon$ depending only on $\varepsilon$.
    In particular, we get $\Pr[\mathcal{E}^{\mathrm{nogiant}}] = o(1)$ as $n \to \infty$.

    We first note that if $d \geq n^{3/4}$, then
    \[
        \frac{\lambda(G)}{\sqrt{\Delta(G)}} \geq \frac{d}{\sqrt{n}} \geq n^{1/4}.
    \]
    Therefore, in this case by \cref{thm:main large lambda} the probability of $\mathcal{E}^{\mathrm{nogiant}}$ is at most $\exp(-c_{\varepsilon}n^{1/4})$, for some constant $c_{\varepsilon}$ depending only on $\varepsilon$.
    Thus, it remains to prove the statement assuming that $d < n^{3/4}$. 
    
    In this regime, we consider the following random process, which we initialize with $\mathbf{S}_1 = \emptyset$. 
    Let $T = \sqrt{\varepsilon n/(16d)} \geq 6$, which denotes the maximum number of steps we make.
    In each step $t \in [T]$, we let $\mathbf{w}_t \in [0, 1]^{V(G) \setminus \mathbf{S}_t}$ be the eigenvector of $G - \mathbf{S}_t$ corresponding to the eigenvalue $\lambda(G-\mathbf{S}_t)$ such that for some $\mathbf{v}_t \in V(G) \setminus \mathbf{S}_t$ we have $\mathbf{w}_t(\mathbf{v}_t) = 1$.
    We reveal the connected component $\mathbf{C}_t$ of $G_p$ containing $\mathbf{v}_t$.
    If $|\mathbf{C}_t| \geq \gamma d$, we terminate the process with success.
    Otherwise, if $\sum_{x \in \mathbf{C}_t} d_{G - \mathbf{S}_t}(x) \geq Td^2$,
    we terminate the process with failure.
    Finally, if neither of the two hold, then set $\mathbf{S}_{t+1} = \mathbf{S}_t \cup \mathbf{C}_t$ and continue.

    To analyse the process, for each $t \in \{0, 1, \dots, T\}$ we let $\mathcal{E}_t^{\mathrm{nowin}}$ and $\mathcal{E}_t^{\mathrm{nofail}}$ denote the events that the process has not terminated with success or failure, respectively, up until and including step $t$.
    We note that $\mathcal{E}_t^{\mathrm{nowin}}, \mathcal{E}_t^{\mathrm{nofail}} \in \sigma(\mathbf{C}_1, \dots, \mathbf{C}_{t})$ and that we have $\Pr[\mathcal{E}^{\mathrm{nogiant}}] \leq \Pr[\mathcal{E}_T^{\mathrm{nowin}}]$.

    We first want to understand the probabilities of success and failure at any given step given that the process has not terminated before this step.
    For that, we fix $t \in \{0, 1, \dots, T-1\}$, let $C_1, \dots, C_t \subseteq V(G)$ and let $\mathcal E$ be the event that $(\mathbf{C}_1, \dots, \mathbf{C}_t) = (C_1, \dots, C_t)$.
    Suppose that $\mathcal E$ holds with positive probability and that $\mathcal E \subseteq \mathcal{E}_t^{\mathrm{nowin}} \cap \mathcal{E}_t^{\mathrm{nofail}}$.
    Define moreover $S_{t'} = \bigcup_{i=1}^{t'-1} C_i$ for each $1 \leq t' \leq t+1$.
    
    We first argue that the average degree of $G - S_{t+1}$ is still roughly $d$.
    Indeed, we have
    \[
        e(G - S_{t+1}) \geq e(G) - \sum_{t'=1}^{t} \sum_{x \in C_{t'}} d_{G-S_{t'}}(x) \geq nd/2 - t Td^2 \geq (1-\varepsilon/8)\cdot nd/2,
    \]
    where we used that $\mathcal E \subseteq \mathcal{E}_t^{\mathrm{nofail}}$ and that $t   \leq T$.
    Consequently, $d(G - S_{t+1}) \geq (1-\varepsilon/8)d$.

    Now, we can apply \cref{lemma:component_of_a_vertex} to bound the probabilities of success and failure at step $t+1$.
    To that end, we first notice that conditioned on the event $\mathcal E$ the graph $G_p - S_{t+1}$ is distributed as $(G - S_{t+1})_p$.
    Moreover, by the above considerations we have $\lambda(G-S_{t+1}) \geq d(G-S_{t+1}) \geq (1-\varepsilon/8)d$ and so $p \geq (1+\varepsilon/2)/\lambda(G - S_{t+1})$.
    Note further that $\mathbf{w}_{t+1}(\mathbf{C}_{t+1}) \geq \mathbf{w}_{t+1}(\mathbf{v}_{t+1}) = 1$ with probability $1$.
    Therefore, by  \cref{lemma:component_of_a_vertex}\ref{it:supercritical sharpness} applied with $\varepsilon/2$ and $C = 1$ we get
    \[
        \Pr\Bigl[\mathcal{E}_{t+1}^{\mathrm{nowin}} \mid \mathcal E\Bigr] \leq \Pr\Bigl[|\mathbf{C}_{t+1}| \leq \gamma'\lambda(G- S_{t+1}) \mid \mathcal E\Bigr] \leq \exp\left(-\frac{\varepsilon^2}{72}\right).
    \]
    Similarly, by \cref{lemma:component_of_a_vertex}\ref{it:not too many edges} applied with $k = T \geq 6$ we get
    \[
        \Pr\Bigl[\neg\mathcal{E}_{t+1}^{\mathrm{nofail}} \mid \mathcal E\Bigr] \leq \Pr\left[|\mathbf{C}_{t+1}| \leq p^{-1} \text{ and } \left( \sum_{x \in \mathbf{C}_{t+1}} d_{G-S_{t+1}}(x) \right) \geq Td^2\mid \mathcal E\right] \leq \exp\left(-\frac{Td}{32}\right).
    \]
    Finally, we notice that since we can partition $\mathcal{E}_t^{\mathrm{nowin}} \cap \mathcal{E}_t^{\mathrm{nofail}}$ by such events $\mathcal E$, we get the same bounds on $\Pr\bigl[\mathcal{E}_{t+1}^{\mathrm{nowin}}| \mathcal{E}_t^{\mathrm{nowin}} \cap \mathcal{E}_t^{\mathrm{nofail}}\bigr]$ and $ \Pr\bigl[\neg \mathcal{E}_{t+1}^{\mathrm{nofail}}| \mathcal{E}_t^{\mathrm{nowin}} \cap \mathcal{E}_t^{\mathrm{nofail}}\bigr]$ respectively.

    To conclude the proof, we note that for the event $\mathcal{E}_T^{\mathrm{nowin}}$ to hold, the process must have either terminated with failure at some step or run all the way up until the end without ever terminating with success.
    Therefore, using that $d \leq n^{3/4}$ we get
    \begin{align*}
        \Pr[\mathcal{E}^{\mathrm{nogiant}}] &\leq \left(\prod_{t=1}^T \Pr\left[\mathcal{E}_{t}^{\mathrm{nowin}} \mid \mathcal{E}_{t-1}^{\mathrm{nofail}} \cap \mathcal{E}_{t-1}^{\mathrm{nowin}}\right] \right)+\left(\sum_{t=1}^T \Pr\left[\neg \mathcal{E}_{t}^{\mathrm{nofail}} \mid \mathcal{E}_{t-1}^{\mathrm{nofail}} \cap \mathcal{E}_{t-1}^{\mathrm{nowin}}\right]\right)\\
        &\leq\exp\left(\left(-\frac{\varepsilon^2}{72}\right) \cdot \sqrt{\frac{\varepsilon n}{16d}} \right)+ \left(\sqrt{\frac{\varepsilon n}{16d}}\right) \cdot \exp\left(-\sqrt{\frac{\varepsilon n}{16d}} \cdot \frac{d}{32}\right)\\
        &\leq \exp\Bigl(-c_{\varepsilon}n^{1/8}\Bigr),
    \end{align*}
    for some constant $c_\varepsilon$ depending only on $\varepsilon$.
\end{proof}

\subsection{Proof of Theorem~\ref{thm:main lambda}}
We now turn to proving \cref{thm:main lambda}, starting with the supercritical regime.
That is, we let $G$ be an arbitrary graph and want to argue that if $p \geq (1+\varepsilon)/\lambda(G)$, then with constant probability $G_p$ contains a component of size $\Omega(\lambda(G))$.
For that we choose a suitable threshold $K$, which will depend on the final probability we want to achieve.
Then, we notice that if $\lambda(G) \geq K\sqrt{\Delta(G)}$, then with probability $1-\exp(-\Omega(K))$ we will get such a large component in $G_p$ by \cref{thm:main large lambda}.
On the other hand, if $\lambda(G) \leq K\sqrt{\Delta(G)}$ then by the Chernoff bound, with very high probability $G_p$ will contain a vertex with degree at least $\Delta(G)/\lambda(G) \geq \lambda(G)/K^2$, which immediately gives us a large connected component.

\begin{proof}[Proof of Theorem~\ref{thm:main lambda}\ref{it:lambda supercritical no assumptions}]
    Fix $\varepsilon, \delta > 0$, let $G$ be a graph and set $p = (1+\varepsilon)/\lambda(G)$.
    Let $\gamma', c$ be the parameters from \cref{thm:main large lambda} applied with $\varepsilon$ and choose $K$ sufficiently large such that $\delta \geq \max \{ \exp(-K/8), \exp(-cK)\}$.
    % \yuval{GPT doesn't like this argument and claims we need to redefine $\delta$ and add a short argument in some case. Don't 100\% understand what it's talking about so recording this here for now} 
    Let moreover $\gamma = \min\{ \gamma', K^{-3} \}$.
    Note that we can assume that $\lambda(G) \geq K^3$, as otherwise $\gamma\lambda(G) \leq 1$ and $G_p$ trivially contains a component of size at least one.

    Now, if $\lambda(G) \geq K\sqrt{\Delta(G)}$, then by \cref{thm:main large lambda} with probability $1-\exp(-cK) \geq 1 - \delta$ there is a component of size at least $\gamma'\lambda(G) \geq \gamma\lambda(G)$ in $G_p$.
    On the other hand, if $\lambda(G) \leq K\sqrt{\Delta(G)}$ then we consider a vertex $v \in V(G)$ with $d_G(v) = \Delta(G) \geq \lambda(G)^2/K^2$.
    By the Chernoff bound (see e.g.\ \cite[Theorem 2.1]{JLR}) we immediately get that%\yuval{Should the first term below be $\Pr\bigl[d_{G_p}(v) < \gamma'\lambda(G)\bigr]$, i.e. with $\gamma'$ rather than $\gamma$?}
    \[
        \Pr\bigl[d_{G_p}(v) < \gamma\lambda(G)\bigr] \leq \exp\left(-\frac{\lambda(G)}{8K^2}\right) \leq \exp(-K/8) \leq \delta,
    \]
    which in particular implies that, with probability at least $1-\delta$, $G_p$ contains a component of size at least $\gamma\lambda(G)$.
\end{proof}

For the subcritical phase of \cref{thm:main lambda}, instead of looking at the component of each fixed vertex $v \in V(G)$, we make the following more global argument.
We first show that if $p \leq (1-\varepsilon)/\lambda(G)$, then the expected number of paths in $G_p$ is $O(\ab{G})$.
Next, we notice that if $G_p$ had a component of size $C\sqrt{\ab{G}}$, then it would immediately give us $C^2\ab{G}/2$ paths, one for each pair of vertices in this large component.
Hence we can bound the probability that $G_p$ contains a large component using Markov's inequality.
We remark that a similar argument was already used in \cite{chunghorn}.

\begin{proof}[Proof of Theorem~\ref{thm:main lambda}\ref{it:lambda subcritical no assumptions}]
    Fix $\varepsilon, \delta > 0$ and let $L = \sqrt{2/(\delta\varepsilon)}$.
    Moreover, let $p \leq (1-\varepsilon)/\lambda(G)$ and let $N$ be the random variable counting the number of distinct paths in $G_p$, where we count vertices as paths of length zero.
    To prove the statement, we want to combine suitable lower and upper bounds on $\mathbb{E}[N]$.

    For the former, since any pair $x, y$ of (not necessarily distinct) vertices lying in the same component of $G_p$ defines at least one distinct path, we get
    \[
        \mathbb{E}[N] \geq \Pr[G_p \text{ has a component of order at least } L\sqrt{\ab{G}}] \cdot L^2\ab G/2.
    \]
    On the other hand, to get an upper bound on $\mathbb{E}[N]$, we notice that for any $k \geq 0$ the number of paths of length $k$ in $G$ is at most $\mathds{1}^TA_G^k\mathds{1}$, the number of walks of length $k$, where $\mathds{1}$ is the all-ones vector.
    Furthermore, by the Cauchy--Schwarz inequality we have $\mathds{1}^TA_G^k\mathds{1} \leq \lVert\mathds{1}\rVert_2\lVert A_G^k\mathds{1}\rVert_2 \leq \lambda(G)^k \ab{G}$.
    Since a path of length $k$ in $G$ survives into $G_p$ with probability $p^k$, we get that
    \[
        \mathbb{E}[N] \leq \sum_{k \geq 0} p^k \cdot \lambda(G)^k \ab{G}\leq \ab{G} \sum_{k \geq 0} (1-\varepsilon)^k = \frac{\ab{G}}{\varepsilon}.
    \]
    By combining the lower and upper bounds and rearranging the terms, we get
    \[
        \Pr[G_p \text{ has a component of order at least } L\sqrt{\ab{G}}] \leq \frac{2}{L^2\varepsilon} = \delta,
    \]
    as desired.
\end{proof}

\section{Concluding remarks}\label{section:concluding remarks}
In this paper, we determined that, under the definition explained in the introduction, $1/\lambda(G)$ is the critical probability for all finite graphs $G$.
In particular, we show that for any finite graph $G$, if $p \geq (1+\varepsilon)/\lambda(G)$, then with probability arbitrarily close to $1$ the graph $G_p$ contains a component of size $\Omega(\lambda(G))$.
As the examples of $K_{s,t}$ and a disjoint union of cliques illustrate, the statement is the best one can hope to get for an arbitrary graph $G$.

A natural question one can still ask is under what additional assumptions on the graph $G$ we get a component of size $\Omega(\ab{G})$ in $G_p$.
This question was addressed in the work of Chung, Horn and Lu~\cite{chunghorn}, yet the conditions they put on the graph $G$ are rather restrictive.
Moreover, the way they measure the size of a component is not by the number of vertices in it, but by a function of the degree sequence, which makes it hard to translate their statements to our setting.

\begin{question}\label{question:omega n supercritical}
    What natural conditions on $G$ imply that if $p \geq (1+\varepsilon)/\lambda(G)$ then with constant probability $G_p$ contains a component of size $\Omega(\ab{G})$?
\end{question}

The main difficulty in finding such a condition is that $\lambda(G)$ can be governed by a very small part of the graph.
In particular, $G$ might contain an induced subgraph $H$ on $(1-o(1))\ab{G}$ vertices such that $\lambda(H) \ll \lambda(G)$, and in particular $H_p$ will behave subcritically even if $p = (1+\varepsilon)/\lambda(G)$.
Thus, simply assuming some notion of ``robust connectivity" would probably not be sufficient, as it will not see such a ``localization of the spectral radius''.

Turning to the subcritical regime, \cref{thm:main large lambda} shows that if $p \leq (1-\varepsilon)/\lambda(G)$ then a.a.s.\ the largest component of $G_p$ will be of size at most $O(\sqrt{\ab{G} \Delta(G)} \log {\ab{G}}/\lambda(G))$.
By taking $G = K_{s, t}$ one can see that this bound is optimal if, roughly speaking, $\lambda(G) \gg \sqrt{\ab{G}}$.
For example, if $G = K_{s, s^3}$, for which we have $\Delta(G) = s^3$ and $\lambda(G) = s^2$, then one can show that for any fixed $\varepsilon$ as $s \to \infty$ and $p = (1-\varepsilon)/\lambda(G)$ a.a.s.\ $G_p$ will contain a connected component of size $\Omega(s\log s)$.\footnote{This is because there will be $\Omega(\log s)$ vertices on the smaller side lying in the same component of $G_p$, each having roughly $s$ neighbours in the larger side.}

However, the bound of $O(\sqrt{\ab{G} \Delta(G)} \log {\ab{G}}/\lambda(G))$ is not tight in all regimes of $\Delta(G)$ and $\lambda(G)$.
\cref{thm:main lambda} shows that if $\lambda(G) \approx \sqrt{\Delta(G)}$, then the $\log {\ab{G}}$-factor is not necessary.
If $\lambda(G)$ is a constant and $\ab{G} \to \infty$, then a similar argument to the one used to prove \cref{thm:main lambda}, but counting paths of length at least $\Omega(\log {\ab{G}})$, shows that if $p \leq (1-\varepsilon)/\lambda(G)$ then a.a.s.\ the largest component of $G_p$ will have size at most $\ab{G}^{1/2-\delta}$, for some $\delta > 0$ depending on $\varepsilon$ and $\lambda(G)$.\footnote{On the other hand, for any $\delta'> 0$ if we take $G$ to be a $d$-ary tree, for a sufficiently large constant $d$, then $G_p$ is likely to contain a component of size $\ab{G}^{1/2 -\delta'}$. This follows from the fact that the spectral radius of the $d$-ary tree approaches $2\sqrt{d}$ as the number of layers goes to infinity.}
% \yuval{I don't understand this sentence or the footnote\textbf{MC: } This is a bit underexplained but I think with careful parsing it can also be understood. Im okay with writing more here or leaving it as is. \textbf{YW} I have reread it and think it's OK}
Finally, if $\lambda(G)=\Delta(G)$ and $G$ is connected, then $G$ is $\lambda(G)$-regular, and in particular we get a $O(\log {\ab{G}})$ bound on the size of the largest connected component in $G_p$.

It would be interesting to know the correct upper bound for all regimes of $\Delta(G)$ and $\lambda(G)$.
\begin{question}\label{question:tight subcritical}
    For a graph with given $\Delta(G)$ and $\lambda(G)$, what is the tight upper bound for the typical size of the largest component of $G_p$, when $p \leq (1-\varepsilon)/\lambda(G)$?
\end{question}

An interesting special case of \cref{question:tight subcritical} would be to consider graphs $G$ which are ``close to regular'', in the sense that $\Delta(G) \leq (1+\delta)\lambda(G)$, for some small $\delta$, possibly being a function of $\ab{G}$. 

Finally, it would be interesting to say more about the structure of the large components in the supercritical regime. In the binomial random graph $G(n,p)$, it is well-known that when $p \geq (1+\varepsilon)/n$, the giant component not only has $\Omega(n)$ vertices, but in fact has a path and a cycle of length $\Omega(n)$ a.a.s. The technique of Krivelevich and Sudakov \cite{MR3085765}, which helped inspire our work, implies that if $G$ is any $d$-regular graph and $p \geq (1+\varepsilon)/d$, then a.a.s.\ $G_p$ has a path of length $\Omega(d)$, and Krivelevich--Samotij \cite{MR3177525} showed that under the same assumptions one obtains a cycle of length $\Omega(d)$. We reiterate a conjecture of Krivelevich and Samotij \cite{MR3177525} (see also \cite{2308.10267}), stating that such a result should still be true if we only assume that the average degree is $d$.
% \yuval{Can one also get a cycle from a cheap sprinkling argument? In $G_{n,p}$ yes but I'm not sure how to do it in $G_p$ in general \textbf{MC: } My guess is yes but only more or less cheap. At $2+eps/d$ it comes easily from the dfs-tree stuff but below i think you need to work a little. It should be easy to show that components of size $\Omega(d)$ usually induce a subgraph of $G$ with average degree $d$. Then one can sprinkle in this unless it is somehow weirdly star shaped or sth like that but im sure there are many ways to salvage that} It would be very interesting to extend such a statement to graphs of average degree $d$.
\begin{conjecture}[\cite{MR3177525,2308.10267}]\label{conj:long cycle}
For every $\varepsilon>0$ there exist $d_0,\gamma>0$ such that the following holds. If $G$ is a graph of average degree $d\geq d_0$ and $p\geq (1+\varepsilon)/d$, then $G_p$ contains a cycle of length at least $\gamma d$ a.a.s.\ as $\ab G \to \infty$.
% \yuval{this previously said $\ab G \to \infty$ but GPT convinced me we probably need $d\to\infty$}
% \yuval{Does this conjecture already appear somewhere? The Krivelevich--Samotij paper has something similar}
\end{conjecture}
We remark that in \cite{2607.02483}, it is proved that if $p \geq C/d$, then $G_p$ a.a.s.\ contains a cycle of length at least $(1-\gamma)d$, where $\gamma \to 0$ as $C \to \infty$. It is possible that by combining the techniques of the present paper with those of \cite{2607.02483}, one could resolve \cref{conj:long cycle}.

\paragraph{Acknowledgments:} We thank Alp M\"uyesser for bringing this problem to our attention and for very helpful discussions in the early phases of this project.

\paragraph{Statement of AI use:} ChatGPT 5.5 suggested the proof of \cref{lemma:weight_subcritical}, but all of the other ideas in the paper, as well as all of the writing, are due entirely to the authors. ChatGPT 5.6 was used to proofread this paper. 

\bibliographystyle{yuval}
{{\bibliography{lib}}}
\end{document}